\documentclass[ centertags, reqno]{amsart}           
\usepackage{amssymb}
\usepackage{times}
\usepackage[mathcal]{euscript} 
\newtheorem{thm}{Theorem}%[section]
\newtheorem{lemma}{Lemma}
\newtheorem{cor}{Corollary}
\newtheorem{prop}{Proposition}

\theoremstyle{definition}

\theoremstyle{remark}
\newtheorem{rem}{Remark} 
\newtheorem{ex}{Example}

\newcommand{\mr}{{\mathbb R}}

\newcommand{\mn}{{\mathbb N}}
\newcommand{\mz}{{\mathbb Z}}
\newcommand{\mc}{{\mathbb C}}

\newcommand{\mh}{{\mathbb H}}

\renewcommand{\rho}{\varrho}

\renewcommand{\Im}{\operatorname{Im}}
\renewcommand{\Re}{\operatorname{Re}}

\newcommand{\supp}{\operatorname{supp}}

\newcommand{\hil}{\mathcal{H}}

\newcommand{\num}{\operatorname{Num}}
\newcommand{\dom}{\operatorname{Dom}}
\renewcommand{\ker}{\operatorname{Ker}}

\begin{document}
   
\title[Absence of boundary eigenvalues for n.s.a.  Schr\"odinger operators]{Absence of eigenvalues of non-selfadjoint Schr\"odinger operators on the boundary of their numerical range}

\author[M. Hansmann]{Marcel Hansmann}
\address{Faculty of Mathematics\\ 
Chemnitz University of Technology\\
Chemnitz\\
Germany.}
\email{marcel.hansmann@mathematik.tu-chemnitz.de}

\begin{abstract}
We use a classical result of Hildebrandt to establish simple conditions for the absence of eigenvalues of non-selfadjoint discrete and continuous Schr\"odinger operators on the boundary of their numerical range.
\end{abstract} 


\maketitle 

\section{Introduction}  
     
The recent interest in spectral properties of non-selfadjoint Schr\"o\-dinger operators has already lead to a variety of new results, both in the continuous and discrete settings. For operators in $L^2(\mr^\nu)$ this includes, e.g., bounds on complex eigenvalues \cite{MR1819914, FLS11, MR2651940,Frank11} and Lieb-Thirring type inequalities \cite{MR2260376, MR2540070, MR2559715, MR2596049, Hansmann11}, and similar results were obtained for discrete Schr\"odinger (and Jacobi) operators in $l^2(\mz^\nu)$ as well \cite{MR2163601, MR2367876, MR2481997, HK10}.
 
In this paper, we will focus on a special class of eigenvalues of non-selfadjoint discrete and continuous Schr\"odinger operators. Namely, we will consider those eigenvalues which are situated on the topological boundary of the numerical range of these operators. As we will show, these eigenvalues are special in the sense that under mild assumptions on the imaginary part of the potential they \textit{cannot} occur.

To indicate the contents of this paper in a little more detail let us consider a Schr\"odinger operator $H=-\Delta+V$ in $L^2(\mr^{\nu})$, with a complex-valued potential $V$ (see Section \ref{sec:continuous} for precise definitions). The numerical range of $H$ is defined as 
$$ \num(H) = \{ \langle Hf,f \rangle: f \in \dom(H), \|f\|=1\}.$$
It is well known that $\num(H)$ is a convex set which, given suitable assumptions on $V$, is contained in a sector in the complex plane. Moreover, the spectrum of $H$ is contained in the closure of the numerical range and so bounds on the numerical range can be used to control the spectrum. We refer to \cite{b_Davies}, Chapter 14.2, for more information on this topic. 
 
The main reason why there will 'usually' be no eigenvalues on the boundary of the numerical range is the fact that these eigenvalues, which in the following we will call \textit{boundary eigenvalues}, do behave like eigenvalues of normal operators. That is, if $\lambda$ is a boundary eigenvalue then
\begin{equation}
  \label{eq:1}
Hf=\lambda f \quad \Leftrightarrow \quad H^*f= \overline{\lambda}f  
\end{equation} 
with the same eigenfunction $f$. In particular, by adding and subtracting these two identities we see that simultaneously $$(-\Delta+\Re(V))f=\Re(\lambda)f \quad \text{and} \quad \Im(V)f=\Im(\lambda)f.$$
This quite restrictive condition will allow  us to prove (using unique continuation)  that boundary eigenvalues can only occur if for some $b \in \mr$ and \textit{every} non-empty open set $\Omega \subset \mr^\nu$ the set 
$\{ x \in \Omega : \Im(V(x))=b\}$  
has positive Lebesgue measure (Theorem \ref{thm:last}). In particular, $\Im(V)$ must be equal to $b$ on a dense subset of $\mr^\nu$ and so, for example, $H$ will have no boundary eigenvalues if $\Im(V)$ is continuous and non-constant. 

For bounded operators, the validity of (\ref{eq:1}) for eigenvalues on the boundary of the numerical range is a classical result of Hildebrandt \cite{MR0200725}. We will see in the next section that his proof, with minor modifications, remains valid in the unbounded case as well. Applications of Hildebrandt's theorem to discrete and continuous Schr\"odinger operators will then be discussed in Section \ref{sec:discrete} and \ref{sec:continuous}, respectively.

\section{Hildebrandt's theorem}\label{sec:Hildebrandt}

Let $Z$ be a closed and densely defined linear operator in a complex separable Hilbert space $(\hil, \langle ., . \rangle)$. We recall that its numerical range $\num(Z)$ (and so its closure $\overline{\num}(Z)$) is a convex set and that, if $\mc \setminus \overline{\num}(Z)$ contains at least one point of the resolvent set of $Z$, then the spectrum of $Z$ (denoted by $\sigma(Z)$)  is contained in $\overline{\num}(Z)$, see \cite{b_Davies}. As  mentioned in the introduction we call an eigenvalue of $Z$ a \textit{boundary eigenvalue} if  it is an element of the topological boundary of the numerical range of $Z$ (denoted by $\partial (\num(Z))$). 
\begin{rem}
In the literature the term boundary eigenvalue is sometimes used with a different meaning. Namely, to denote eigenvalues (of bounded operators) whose absolute value coincides with the spectral radius of the operator. However, usually these eigenvalues are called peripheral eigenvalues.
\end{rem}
An eigenvalue $\lambda$  of $Z$  is called a \textit{normal eigenvalue} if
\begin{equation*}
    \ker(\lambda-Z)=\ker(\overline{\lambda}-Z^*),
\end{equation*}
that is, $f \in \dom(Z)$ and $Zf=\lambda f$ if and only if $f  \in \dom(Z^*)$ and $Z^*f=\overline{\lambda}f$. As indicated above, the analysis of the normal eigenvalues of $Z$ can be reduced to the study of the operators
$$ \Re(Z)=\frac 1 2 (Z+Z^*) \quad \text{and} \quad \Im(Z)=\frac 1 {2i} (Z-Z^*).$$
\begin{rem}
Throughout this article, if not indicated otherwise, the sum of two operators is understood as the usual operator sum with $\dom(Z+Z_0)=\dom(Z) \cap \dom(Z_0)$.
\end{rem}
\begin{prop}\label{prop:1}
Let $f \in \ker(\lambda-Z)\cap \ker(\overline{\lambda}-Z^*)$. Then
$$ \Re(Z)f=\Re(\lambda)f \qquad \text{and} \qquad \Im(Z)f=\Im(\lambda)f.$$
\end{prop}
\begin{proof}
  A short calculation.
\end{proof}
\begin{rem} \label{rem:XXX}
The following facts are easily checked:
\begin{enumerate}
\item[(i)] $\num(\Re(Z)), \num(\Im(Z)) \subset \mr$.
\item[(ii)] If $\dom(Z) \subset \dom(Z^*)$ then $Z=\Re(Z)+i\Im(Z)$ and so $$\num(Z) \subset \num(\Re(Z))+i\num(\Im(Z)).$$
\item[(iii)] If $Z$ is a bounded operator then $\num(Z) \subset \{ \lambda : |\lambda| \leq \|Z\|\}$.
\end{enumerate}
\end{rem}
We continue with Hildebrandt's theorem. 
\begin{thm}\label{unbounded}
  Let $Z$ be a densely defined closed  operator in $\hil$ such that $\dom(Z) \subset \dom(Z^*)$ and let $\lambda$ be a boundary eigenvalue of $Z$. Then
  \begin{equation}
    \label{inclusion}
 \ker(\lambda-Z) \subset \ker(\overline{\lambda}-Z^*).    
  \end{equation}
Furthermore, if $\dom(Z)=\dom(Z^*)$ then $\lambda$ is a normal eigenvalue. 
\end{thm}
\begin{rem}
We should emphasize that the assumption $\dom(Z) \subset \dom(Z^*)$ will really be important in our proof of (\ref{inclusion}), since it assures that every eigenfunction corresponding to a boundary eigenvalue of $Z$ is in the domain of $Z^*$. 
\end{rem}
As mentioned earlier, Hildebrandt \cite{MR0200725} proved this theorem in the case where $Z$ is a bounded operator on $\hil$. On the other hand, the proof of the general case presented below requires only some minor adjustments to the original proof. We start with some preparatory results, most of which are straightforward or well known.
\begin{lemma}\label{lem:num}
Let $\alpha, \beta \in \mc$. Then 
\begin{equation*}
  \num(\alpha Z + \beta) = \alpha \num(Z) + \beta. 
\end{equation*}
Moreover, if $\dom(Z)=\dom(Z^*)$ then
\begin{equation*}
\num(Z^*)=\{ \overline{\lambda}: \lambda \in \num(Z)\}.  
\end{equation*}
\end{lemma}
\begin{lemma}\label{lem:1} 
 Let $A$ be a densely defined, symmetric, non-negative operator in  $\hil$. 
If $f \in \dom({A})$ and $\langle{A}f,f \rangle=0$, then ${A}f=0$.
\end{lemma}
\begin{proof}
Let $B$ denote a non-negative, selfadjoint extension of $A$. Then $B$ has a non-negative square root, so we obtain $ 0= \langle Af,f \rangle = \langle Bf, f \rangle = \| \sqrt{B} f \|^2$. This implies $\sqrt{B}f=0$ and so $Af=Bf=\sqrt{B} \sqrt{B} f =0$.
\end{proof}
We will also need what is sometimes known as the \textit{supporting hyperplane theorem}, see \cite{MR1443208} Theorem 2.4.12. 
\begin{thm}
Let $S$ be a convex set in $\mc$ and let $x \in \partial S$. Then there exists a closed half-plane $\mh$ such that $x \in \partial \mh$ and $S \subset \mh$.
\end{thm}
Now we are prepared for the proof of Theorem \ref{unbounded}: Let  $\lambda \in \partial(\num(Z))$ with $Zf=\lambda f$ for some non-trivial $f \in \dom(Z) \subset \dom(Z^*)$. We need to show that 
\begin{equation}
  \label{eq:AA}
 Z^*f=\overline{\lambda} f. 
\end{equation}

By the supporting hyperplane theorem and Lemma \ref{lem:num} we can find $\theta \in [0,2\pi)$ such that $T:=e^{i\theta}(Z-\lambda)$ satisfies $\num(T)\subset \{ \lambda : \Im(\lambda) \geq 0\}$.  Moreover, we have $Tf=0$. In the following we show that $T^*f=0$, which implies (\ref{eq:AA}).

By construction $\Im(T)$ is densely defined (note that $\dom(\Im(T))$ $=\dom(Z)$), symmetric and non-negative. Since $Tf=0$ we also have
$$  \langle \Im(T) f, f \rangle = \Im( \langle T f, f \rangle ) = 0.$$
So we can apply Lemma \ref{lem:1} to obtain that $\Im(T)f=0$. Since $\Re(T)=T-i\Im(T)$, this implies that $\Re(T)f=0$ as well. Finally, the symmetry of $\Re(T)$ and $\Im(T)$ implies that
\begin{equation*} 
  T^*= (\Re(T)+i\Im(T))^* \supset \Re(T)^*-i\Im(T)^* \supset \Re(T)-i\Im(T).
\end{equation*}
This inclusion shows that $\Re(T)f-i\Im(T)f=T^*f$ and so $T^*f=0$ as desired.

In the preceding part of the proof we have shown (in case $\dom(Z) \subset \dom(Z^*)$) that $\ker(\lambda-Z) \subset \ker(\overline{\lambda}-Z^*)$ if $\lambda$ is a boundary eigenvalue of $Z$.
It remains to show that if $\dom(Z)=\dom(Z^*)$, then also the reverse inclusion is valid and so $\lambda$ is a normal eigenvalue. But in this case Lemma \ref{lem:num} shows that $\overline{\lambda}$ is a boundary eigenvalue of $Z^*$, so  by the first part of the proof $$\ker(\overline{\lambda}-Z^*) \subset \ker(\overline{\overline{\lambda}}-Z^{**}) = \ker(\lambda-Z).$$
This concludes the proof of Theorem \ref{unbounded}.

\section{The discrete Schr\"odinger operator}\label{sec:discrete}
In this section we apply Hildebrandt's theorem to  derive conditions for the absence of boundary eigenvalues for the non-selfadjoint discrete Schr\"o\-dinger operator $J=J_0+D$ on $l^2(\mz^\nu), \nu \geq 1$. Here $J_0$ denotes the discrete Laplacian on $l^2(\mz^\nu)$, i.e.
$$ (J_0 u)(k)=\sum_{l \in \mz^\nu : \: \|l\|_1=1} u(k+l),$$
where $\|l\|_1=\sum_{j=1}^\nu |l_j|$, and $D$ is the operator of multiplication by a bounded function $d : \mz^\nu \to \mc$, i.e. 
$(Du)(k)=d(k)u(k)$. 
\begin{rem}
Recall that $\sigma(J_0)=[-2\nu,2\nu]$ and $\sigma(D)=\overline{\{ d(k) : k \in \mz^\nu \}}$.   
\end{rem}
Since $J_0=J_0^*$ we see that $\Re(J)= J_0 + \Re(D)$ and $\Im(J)=\Im(D)$, where 
$(\Re(D)u)(k)=\Re(d(k))u(k)$ and $(\Im(D)u)(k) = \Im(d(k)) u(k)$. In particular, Remark \ref{rem:XXX} implies that
\begin{equation}
  \num(J) \subset [-2\nu+R_-,2\nu+R_+]+i[I_-,I_+],
\end{equation}
where  $ R_-=\inf \Re(d(k))$, $R_+= \sup \Re(d(k))$, $I_-=\inf \Im(d(k))$ and $I_+=\sup \Im(d(k))$.
\begin{prop}\label{prop:6}
  Let $\lambda$ be a boundary eigenvalue of $J$ with corresponding eigenfunction $u$. Then
  \begin{equation}\label{eq:dd}
 (J_0 + \Re(D))u = \Re(\lambda)u \quad \text{and} \quad \Im(D)u = \Im(\lambda)u.
  \end{equation}
\end{prop}
\begin{proof}
  Apply Hildebrandt's theorem and Proposition \ref{prop:1}.
\end{proof}
The previous proposition provides a first condition for the absence of boundary eigenvalues. We will use the fact that the eigenvalues of the operator $\Im(D)$ are given by $\Im(d(k)), k \in \mz^\nu$.
\begin{cor}\label{cor:9}
  (i) If $a \in \mr$ is not an eigenvalue of $J_0+\Re(D)$, then $J$ has no boundary eigenvalues with real part $a$.\\
  (ii) Let $b \in \mr$. If $\Im(d(k)) \neq b$ for all $k \in \mz^\nu$, then $J$ has no boundary eigenvalues with imaginary part $b$.
\end{cor}
\begin{ex}
The spectrum of $J_0$ is purely absolutely continuous, so Part (i) of the previous corollary implies that the operator $J_0+i\Im(D)$, with a purely imaginary potential, does not have any boundary eigenvalues.
\end{ex}
\begin{rem}
If $\Im(d)$ has a fixed sign (i.e. $\Im(d): \mz^\nu \to \mr_\pm$), then the numerical range of $J$ is contained in $\{ \lambda : \pm \Im(\lambda) \geq 0 \}$. So in this case all real eigenvalues (i.e. eigenvalues in $\mr$) are boundary eigenvalues and the above results, and the results to follow, provide conditions for the absence of such eigenvalues. 
\end{rem}
Corollary \ref{cor:9} can be improved considerably using the following two lemmas. The first one is obvious.
\begin{lemma}\label{jacobi}
Let  $b \in \mr$ and $u \in l^2(\mz^\nu)$. Then $\Im(D)u=b u$ if and only if
$\Im(d(k)) = b$ for all $k \in \supp(u):=\{k \in \mz^\nu : u(k) \neq 0\}$.
\end{lemma}
% \begin{proof}
%   Use the definition of $\Im(D)$.
% \end{proof}
The next lemma shows that the support of an eigenfunction of $J$ must be infinite in 'all' directions.
\begin{rem}   
  Let us agree that throughout this section $k_j$ will denote the $j$th component of $k \in \mz^\nu$. 
\end{rem}
\begin{lemma} \label{newlemma}
 Let $u$ be an eigenfunction of $J$. Then for every $j\in \{1,\ldots,\nu\}$ 
 \begin{equation*}
   \sup \{ k_j : k \in \supp(u) \} = \infty
 \end{equation*}
and
 \begin{equation*}
   \inf \{ k_j : k \in \supp(u) \} = -\infty.
 \end{equation*}
\end{lemma}
\begin{proof}
We only show that $\sup \{ k_1 : k \in \supp(u) \} = \infty$ (all other cases can be proved in exactly the same way). To this end, let us suppose that this supremum is finite and let us set 
$$ M=    \max \{ k_1 : k \in \supp(u) \} \in \mz.$$
In other words, there exists $l \in \mz^{\nu-1}$ such that $u(M,l)\neq 0$ and  for every $n \in \mn$ and every $l'\in \mz^{\nu-1}$ we have $u(M+n,l')=0$. Now let $\lambda$ denote an eigenvalue corresponding to $u$. Then we can evaluate the identity $(Ju)(k)=\lambda u(k)$ at $k=(M+1,l)$ to obtain
$$ \lambda u(M+1,l)=u(M,l)+u(M+2,l)+ \sum_{\|l'-l\|_1=1} u (M+1,l').$$
But here all terms apart from $u(M,l)$ are zero by definition of $M$, so the equation can be satisfied only if $u(M,l)=0$ as well. This leads to a contradiction.
\end{proof}
%The following theorem shows that there will be no boundary eigenvalues of imaginary part $b$ if in at least one 'direction' $\Im(d(k))=b$ for \textit{at most} finitely many $k$.
\begin{thm}\label{thm:new}
Suppose that for some $j \in \{1,\ldots,\nu\}$ one of the following conditions is satisfied:
 \begin{eqnarray*}
\mbox{(i)} \quad &   \sup \{ k_j : \Im(d(k))=b \} < \infty, \\
\mbox{(ii)} \quad &   \inf \{ k_j : \Im(d(k))=b \} > -\infty. 
 \end{eqnarray*}
Then $J$ has no boundary eigenvalues with imaginary part $b$.
\end{thm}
\begin{rem}
Note that $\sup \emptyset = -\infty$ and $\inf \emptyset = \infty$, so Corollary \ref{cor:9}, Part (ii), is a (very) special case of Theorem \ref{thm:new}.
\end{rem}
\begin{proof} 
  Suppose there exists a boundary eigenvalue $\lambda=a+ib$ with corresponding eigenfunction $u$. Then by Proposition \ref{prop:6} and Lemma \ref{jacobi} we have $\Im(d(k))=b$ for all $k \in \supp(u)$. Now apply Lemma \ref{newlemma} to derive a contradiction.
\end{proof}
While Theorem \ref{thm:new} shows that boundary eigenvalues can exist only under very special circumstances, such circumstances can of course occur. For instance, it is easy to provide examples of $J$ having boundary eigenvalues if $\Im(d)$ is constant on $\mz^\nu$ (since in this case $J$ is just a selfadjoint Schr\"odinger operator shifted by a complex constant). The next example shows that boundary eigenvalues can also occur if $\Im(d)$ is non-constant.
\begin{ex}\label{ex:counter}
Choose $\Re(d)$ such that $J_0+\Re(D)$ has an eigenvalue $a \in \mr$ whose corresponding eigenfunction $u$ has at least one zero (this is always possible by starting with $u$ and constructing $\Re(d)$ appropriately). For $b > 0$ define $\Im(d)$ as  
$$ \Im(d(k)) = \left\{ 
  \begin{array}{cl}
    b & , \text{ if } k \in \supp(u) \\
    0 & , \text{ if } k \notin \supp(u).
  \end{array}\right. $$
Then $\Im(D)u=bu$ and so $a+ib$ is an eigenvalue of $J$. Since the numerical range of $J$ is contained in $\{ \lambda : 0 \leq \Im(\lambda) \leq b \}$ this eigenvalue is a boundary eigenvalue. %Here $\db=\mz \setminus \operatorname{supp}(u)$ is finite and does not contain two consecutive integers.
\end{ex}   
We continue with two corollaries of Theorem \ref{thm:new} which provide simple conditions for the absence of non-real boundary eigenvalues.
\begin{cor}\label{cor:X}
  Suppose that for some $j \in \{1,\ldots,\nu\}$ and $n \to \infty$ (or $-\infty$) we have
  \begin{equation}
    \label{eq:xy}
 \sup |\Im(d(k_1,\ldots,k_{j-1},n,k_{j+1},\ldots,k_\nu))| \to 0,     
  \end{equation}
the supremum being taken over all $(k_1,\ldots,k_{j-1},k_{j+1},\ldots,k_\nu) \in \mz^{\nu-1}$.
Then any boundary eigenvalue of $J$ must be real.
\end{cor}
\begin{rem}
If $\nu =1$ then (\ref{eq:xy}) has to be understood as $ \Im(d(n)) \to 0$ for $n \to \infty$ (or $-\infty$). 
\end{rem}
\begin{proof}[Proof of Corollary \ref{cor:X}]
  We only consider the case $j=1$ and $n \to \infty$. By assumption, for every  $b \in \mr \setminus \{0\}$ we can find $n_b$ such that for all $n > n_b$ 
$$ \sup |\Im(d(n,k_2,\ldots,k_\nu))| < |b|,$$
the supremum being taken as above. In particular, this shows that $$\sup \{k_1 : \Im(d(k))=b\} \leq n_b$$ and we can apply Theorem \ref{thm:new} to conclude that there are no boundary eigenvalues with imaginary part $b$.
\end{proof}
Let us also state the following special case of Corollary \ref{cor:X}.
\begin{cor} \label{cor:XX}
If $\Im(d(k)) \to 0$ for $\|k\|_1 \to \infty$, then any boundary eigenvalue of $J$ must be real.
\end{cor}
% \begin{rem}
% Note that Corollary \ref{cor:XX} is indeed weaker than Corollary \ref{cor:X}. For instance, the potential $\Im(d(k_1,k_2))=1/(1+|k_1|)$ satisfies (\ref{eq:xy}) for $j=1$ but does not converge to $0$ for $\|k\|_1 \to \infty$.
% \end{rem} 
\begin{rem}\label{rem:app}
  In many applications one is interested in potentials $d$ with $\lim_{\|k\|_1 \to \infty} d(k)=0$. In this case $J$ is a compact perturbation of $J_0$ and so Weyl's theorem implies that the spectrum of $J$ consists of $[-2\nu,2\nu]$ (the essential spectrum) and a possible set of isolated eigenvalues which can accumulate at $[-2\nu,2\nu]$ only. From Corollary \ref{cor:XX} we now know that none of the non-real eigenvalues of $J$ will be a boundary eigenvalue. 
\end{rem}

In the remaining part of this section we restrict ourselves to the one-dimensional case, i.e.
\begin{equation}
  \label{eq:100}
  (Ju)(n)=u(n-1)+d(n)u(n)+u(n+1), \quad n \in \mz, \quad u \in l^2(\mz).
\end{equation} 
Note that in this case a solution $u$ of the equation $Ju=\lambda u$ is uniquely determined by its values on two consecutive integers $m$ and $m+1$. In particular, if $u(m+1)=u(m)=0$, then $u$ must be identically zero (this fact is sometimes referred to as \textit{unique continuation principle}). As compared to the higher-dimensional case (where non-zero eigenfunctions might vanish on arbitrarily large connected components) these facts will allow us to strengthen our results on the boundary eigenvalues of $J$ considerably. For instance, the next theorem shows that all boundary eigenvalues will have the same imaginary part and that boundary eigenvalues can exist only if the imaginary part of the potential is of a very special form.

\begin{thm}\label{imagin}($\nu=1$) 
If $J$ has a boundary eigenvalue with imaginary part $b$, then 
\begin{enumerate}
\item[(i)] for every $n \in \mz$ we have $$\{ \Im(d(n)), \Im(d(n+1))\} \cap \{b\} \neq \emptyset,$$ 
\item[(ii)] all boundary eigenvalues of $J$ will have imaginary part $b$.
\end{enumerate}
\end{thm}
An immediate corollary of this theorem and Corollary \ref{cor:X} is the following result.
\begin{cor}\label{cor:imagin}($\nu=1$)
Suppose that $\Im(d(n)) \to 0$ for $n \to + \infty (\text{or} -\infty)$ and that there exists $m \in \mz$ with $$\Im(d(m)) \neq 0 \quad \text{and} \quad  \Im(d(m+1)) \neq 0.$$
Then $J$ has no boundary eigenvalues.
\end{cor}
The next lemma will be needed in the proof of Theorem \ref{imagin}.
\begin{lemma}($\nu=1$)
Let $(\Im(d(n)))_{n \in \mz}$ be of the form
\begin{equation}
  \label{eq:N}
 (\ldots,b_1,b_2,b_1,b_2,b_1,b_2,\ldots)
\end{equation} 
for some $b_1 \neq b_2$. Then $J$ has no boundary eigenvalues.
\end{lemma}
\begin{proof} 
If $J$ would have a boundary eigenvalue $\lambda$ and $u$ would denote a corresponding eigenfunction, then Proposition \ref{prop:6} and Lemma \ref{jacobi} would imply that $\supp(u)\subset 2\mz$ (or $\supp(u) \subset 2\mz+1$). But then we could choose $n=2m+1$ (or $n=2m$) in the difference equation 
$$ u(n-1)+d(n)u(n)+u(n+1)= \lambda u(n)$$
to obtain that for $m \in \mz$
$$ u(2m)+u(2(m+1))=0 \qquad (\text{ or } u(2(m-1)+1)+u(2m+1)=0).$$
This would imply that the absolute value of $u$ is constant and non-zero on $2\mz$ (or $2\mz+1$), so $u$ would not be in $l^2(\mz)$.  
\end{proof}
\begin{proof}[Proof of Theorem \ref{imagin}]
Let $a+ib$ be a boundary eigenvalue of $J$ with corresponding eigenfunction $u$.\\
(i) Assume there exists $m\in \mz$ with $\Im(d(m))\neq b$ and $\Im(d(m+1)) \neq b$. Then Proposition \ref{prop:6} and Lemma \ref{jacobi} would imply that $u(m)=u(m+1)=0$. But then by unique continuation $u$ can satisfy the equation $Ju=\lambda u$ only if it is identically zero, a contradiction.\\ 
(ii) Part (i) implies that an additional boundary eigenvalue $c+id$, with $d\neq b$, could exist only if $(\Im(d(n)))_{n \in \mz}$ would be of the form (\ref{eq:N}) (with $b_1=b, b_2=d$). On the other hand, we already know from the corresponding lemma that in this case $J$ would have no boundary eigenvalues. 
\end{proof}
Now that we know that all boundary eigenvalues of $J$ will have the same imaginary part, let us try to obtain a little more information on the real parts of these eigenvalues. Our aim is to show that, under certain assumptions, the real part of a boundary eigenvalue of $J_0+D$ cannot lie below or above the essential spectrum of $J_0+\Re(D)$. First, however, let us consider an example which shows that this need not always be true.
\begin{ex}\label{ex:counter2} ($\nu=1$)
Choose $\Re(d)$ such that $J_0+\Re(D)$ has an eigenvalue $a$ below its essential spectrum and let $u$ denote the corresponding eigenfunction. Then standard oscillation theory (see, e.g., \cite{MR1711536}) implies that $u$ will have only finitely many zeros. In complete analogy  to Example \ref{ex:counter} we can define
$$ \Im(d(n)) = \left\{ 
  \begin{array}{cl}
    b & , \text{ if } n \in \supp(u) \\
    0 & , \text{ if } n \notin \supp(u)
  \end{array}\right. $$ 
to obtain a Schr\"odinger operator with boundary eigenvalue $a+ib$. Note that here the set $\{n \in \mz: \Im(d(n)) \neq b\}$ is finite.
\end{ex} 
\begin{prop}\label{prop:jb}($\nu=1$)
Let $a+ib$ be a boundary eigenvalue of $J$. Moreover, assume that $\Im(d(n)) \neq b$ for infinitely many $n \in \mz$. Then 
$$ \inf \sigma_{ess}(J_0+\Re(D)) \leq a \leq \sup \sigma_{ess}(J_0+\Re(D)).$$
\end{prop}
\begin{proof}
By Proposition \ref{prop:6} there exists $u \in l^2(\mz)$ with $(J_0+\Re(D))u=au$ and $\Im(D)u=b u$.  Furthermore, the assumption and Lemma \ref{jacobi} imply that $u(n)=0$ for infinitely many $n$. On the other hand, as mentioned in Example \ref{ex:counter2}, each eigenfunction of $J_0 + \Re(D)$ corresponding to an eigenvalue below or above the essential spectrum has only a finite number of zeros.
\end{proof}
Finally, let us consider the case where $d(n) \to 0$ for $|n| \to \infty$ (and so $\sigma_{ess}(J)=[-2,2]$, see Remark \ref{rem:app}). From Corollary \ref{cor:XX} (and \ref{cor:imagin}) we know that in this case all boundary eigenvalues must be real and that there will be no boundary eigenvalues  if $\Im(d(m))\neq 0$ and $\Im(d(m+1))\neq 0$ for some $m \in \mz$. If this condition is not satisfied, the next theorem might be useful.  
\begin{thm}($\nu=1$)
Suppose that $d(n) \to 0$ for $|n| \to \infty$ and that    
\begin{enumerate}
\item[(i)]  $\{ \Im(d(n)), \Im(d(n+1))\} \cap \{0\} \neq \emptyset$ for every $n \in \mz$,  
\item[(ii)] $\Im(d(n)) \neq 0$ for infinitely many $n \in \mz$. 
\end{enumerate}
Then $J$ has no boundary eigenvalues if 
\begin{equation}
  \label{eq:de}
\sum_k |k| |\Re(d(k))| < \infty.
\end{equation}
\end{thm}
\begin{proof} 
 We only need to show that $J$ has no real boundary eigenvalues. To this end, note that (\ref{eq:de}) implies that $J_0+\Re(D)$ has only finitely many eigenvalues in $\mr \setminus [-2,2]$ (see \cite{MR1711536} Theorem 10.4). But then \cite{Damanik05}, Theorem 2, implies that $J_0+\Re(D)$ will have no eigenvalues in $[-2,2]$. However, Proposition \ref{prop:jb} shows that any real boundary eigenvalue $\lambda$ of $J$ will be an eigenvalue of $J_0+\Re(D)$ satisfying $\lambda \in [-2,2]$, so no such eigenvalues can exist. 
\end{proof}
\begin{rem}
It would be interesting to know whether some of the above results (like Theorem \ref{imagin}) have analogs in the higher-dimensional case, or whether the absence of unique continuation will prevent such analogs. For the moment, we leave this as an open problem.
\end{rem}

\section{The continuous Schr\"odinger operator}\label{sec:continuous} 
   
In this final section we consider the consequences of Hildebrandt's theorem for the absence of boundary eigenvalues of Schr\"odinger operators $-\Delta + V$  in $L^2(\mr^\nu), \nu\geq 1$. To provide a precise definition of these operators we make the following (rather abstract) assumption on the measurable function $V: \mr^\nu \to \mc$.
\begin{enumerate}
\item[(A1)] The sesquilinear form
\begin{eqnarray*}
  \mathcal{E}_V(f,g) &=& \int_{\mr^\nu} V(x)f(x)\overline{g(x)} dx, \\
 \dom(\mathcal{E}_V) &=& \{ f \in L^2(\mr^\nu): V |f|^2 \in L^1(\mr^\nu)\}
\end{eqnarray*}
is $\mathcal{E}_0$-bounded  with form-bound $<1$, where 
\begin{eqnarray*}
  \mathcal{E}_0(f,g) = \langle \nabla f, \nabla g \rangle, \quad \dom(\mathcal{E}_0)= H^{1,2}(\mr^\nu).
\end{eqnarray*}
\end{enumerate}
Given this assumption the form $\mathcal{E}=\mathcal{E}_0+\mathcal{E}_V, \dom(\mathcal{E})=\dom(\mathcal{E}_0),$ is densely defined, closed and sectorial, 
 so by the first representation theorem (\cite{kato}, Theorem VI.2.1) we can uniquely associate to $\mathcal{E}$ an $m$-sectorial operator $H=:-\Delta \dotplus V$. The numerical range of $H$ will be contained in a sector $\{ \lambda : |\arg(\lambda-\gamma)| \leq \alpha\}$ for some $\gamma \in \mr$ and $\alpha \in [0,\pi/2)$ (see \cite{b_Davies}, Chapter 14.2, for more precise bounds on the numerical range).
\begin{rem}
Note that assumption (A1) can only be satisfied if $V \in L^1_{loc}(\mr^\nu)$. Moreover, (A1) is satisfied if 
 $V \in L^{p}(\mr^\nu)+L^\infty(\mr^\nu)$ (where $p = \nu/2$ if $\nu \geq 3$, $p>1$ if $\nu=2$ and $p=1$ if $\nu=1$), or if
  $|V|$ is in the Kato-class (see \cite{MR670130}). For more general conditions we refer to \cite{MR2198326}.
\end{rem}
To apply Hildebrandt's theorem we have to make sure that $\dom(H)$ is a subset of $\dom(H^*)$. This requires an additional assumption on the imaginary part of the potential.
\begin{enumerate}
\item[(A2)] $ \quad \dom(H) \subset \{ f \in L^2(\mr^\nu) : \Im(V)f \in L^2(\mr^\nu)\}$. 
\end{enumerate}
\begin{rem}
 In other words, (A2) means that $\dom(H)$ is a subset of the domain of the multiplication operator $M_{\Im(V)}$, defined as 
$$M_{\Im(V)}f=\Im(V)f, \quad \dom(M_{\Im(V)})= \{ f \in L^2 : \Im(V)f \in L^2\}. $$
Since the precise domain of $H$ is often quite difficult to establish, this assumption  is even more abstract than (A1). However, since $\dom(H)$ will always be contained in $H^{1,2}(\mr^\nu)$, the Sobolev embedding theorems show that (A2) will be satisfied if $\Im(V) \in L^p(\mr^\nu)+L^\infty(\mr^\nu)$ (where $p=\nu$ if $\nu\geq 3$, $p>2$ if $\nu=2$ and $p=2$ if $\nu=1$).
\end{rem}
In the following lemma $-\Delta \dotplus \Re(V)$ denotes the selfadjoint lower semibounded  operator corresponding to the closed, semibounded form $\mathcal{E}_0+\mathcal{E}_{\Re(V)}$ defined on $\dom(\mathcal{E})$.
\begin{rem}
  Recall that $\Re(H)=\frac 1 2 (H+H^*)$ and $\Im(H)= \frac 1 {2i} (H-H^*)$, both operators being defined on $\dom(H)  \cap \dom(H^*)$. 
\end{rem}
\begin{lemma}\label{lem:form}
Assume (A1) and (A2). Then the following holds:\\
(i)  $\dom(H) \subset \dom(H^*)$ and $H^*=H-2iM_{\Im(V)}$ on $\dom(H)$.\\ 
(ii) $\dom(\Re(H))=\dom(H)$ and $\Re(H)$  is a restriction of $-\Delta \dotplus \Re(V)$. \\
(iii) $\dom(\Im(H))=\dom(H)$ and $ \Im(H)$ is a restriction of $M_{\Im(V)}$.
\end{lemma}
Let us recall some facts about the relation between $H$ and $\mathcal{E}$ which will be needed in the proof of Lemma \ref{lem:form}:
\begin{enumerate}
\item[(i)] $\dom(H) \subset \dom(\mathcal{E})$ and $\mathcal{E}(f,g)=\langle Hf,g \rangle$ for all $f \in \dom(H)$ and $g \in \dom(\mathcal{E})$,
%\item[(ii)] $\dom(H)$ is a core of $\mathcal{E}$,
\item[(ii)] if $f \in \dom(\mathcal{E})$, $h \in L^2$ and $\mathcal{E}(f,g)=\langle h,g \rangle$ for all $g$ belonging to a core of $\mathcal{E}$, then $f \in \dom(H)$ and $Hf=h$. 
\end{enumerate}
Moreover, $H^*$ is the $m$-sectorial operator associated to the adjoint form ${\mathcal{E}^*}$ given by $${\mathcal{E}^*}(f,g)= \overline{\mathcal{E}(g,f)}, \quad \dom({\mathcal{E}^*})=\dom(\mathcal{E}).$$ 
\begin{proof}[Proof of Lemma \ref{lem:form}]
(i) A short computation shows that for $f,g \in \dom(\mathcal{E})=\dom(\mathcal{E}^*)$ we have
$ \mathcal{E^*}(f,g)= \mathcal{E}(f,g)-2i \mathcal{E}_{\Im(V)}(f,g)$, 
so if $f \in \dom(H)$ and $g \in \dom(\mathcal{E}^*)$ we obtain
\begin{eqnarray*}
  0 = \mathcal{E}(f,g) - \langle Hf,g \rangle = \mathcal{E}^*(f,g)+2i\mathcal{E}_{\Im(V)}(f,g) - \langle Hf ,g \rangle,
\end{eqnarray*}
which implies that $ \mathcal{E}^*(f,g)= \langle Hf-2iM_{\Im(V)}f, g \rangle$. Here $Hf-2iM_{\Im(V)}f$ is in $L^2$ by assumption (A2). Since $g$ was arbitrary this implies that $f \in \dom(H^*)$ and $H^*f= Hf-2iM_{\Im(V)}f$. \\% In exactly the same way we obtain that $\dom(H^*) \subset \dom(H)$.\\
(ii) From (i) we know that   $\dom(\Re(H))=\dom(H)$ and $\Re(H)=H-iM_{\Im(V)}$, so for $f \in \dom(H)$ and $g \in \dom(\mathcal{E})$ we obtain $\langle \Re(H) f ,g \rangle  = \mathcal{E}_0(f, g) + \mathcal{E}_{\Re(V)}( f, g )$. But this implies that $f \in \dom(-\Delta \dotplus \Re(V))$ and that $\Re(H)f=(-\Delta \dotplus \Re(V))f$.\\
(iii) This is an immediate consequence of (i).
\end{proof}
We are finally prepared to state a first  result on the boundary eigenvalues of $H$. It is a direct consequence of Lemma \ref{lem:form}, Theorem \ref{unbounded} and Proposition \ref{prop:1}.   
  \begin{prop}\label{prop:2}
Assume (A1) and (A2). Let $\lambda$ be a boundary eigenvalue of $H$ with corresponding eigenfunction $f$. Then
$$ (-\Delta \dotplus \Re(V))f=\Re(\lambda)f \qquad \text{and} \qquad M_{\Im(V)}f=\Im(\lambda)f.$$
  \end{prop} 
In the following corollary we use the fact that $b$ is an eigenvalue of $M_{\Im(V)}$ iff the set $\{x : \Im(V(x))=b\}$ has non-zero Lebesgue measure.
  \begin{cor}\label{cor:4}
Assume (A1) and (A2). Then the following holds:\\
(i) If $a \in \mr$ is not an eigenvalue of $-\Delta \dotplus \Re(V)$, then $H$ has no boundary eigenvalues with real part $a$.\\
(ii) Let $b \in \mr$. If the set $\{ x : \Im(V(x))=b \}$ has Lebesgue measure zero, then $H$ has no boundary eigenvalues with imaginary part $b$.
      \end{cor} 
\begin{ex}
The spectrum of $-\Delta$ %(i.e., the selfadjoint operator associated to the form $\mathcal{E}_0$) 
is equal to $[0,\infty)$ and purely absolutely continuous. Part (i) of the previous corollary thus shows that the operator $-\Delta \dotplus i \Im(V)$ has no boundary eigenvalues.        
\end{ex}
\begin{rem}
  If $\Im(V)$ has a fixed sign then all real eigenvalues of $H$ are boundary eigenvalues, so in this case the results discussed in this section can be used to show the absence of these eigenvalues.
\end{rem} 
Similar to the discrete case, we can strengthen the above results using the following unique continuation result (\cite{jerison}, Thm. 6.3 and Rem. 6.7). Here, 
$$ H_{loc}^{k,q}(\mr^\nu)= \left\{ f \in L_{loc}^q(\mr^\nu) : \frac {\partial^\alpha}{\partial x^\alpha} f \in L_{loc}^q(\mr^\nu), \forall \alpha, |\alpha|\leq k \right\}$$
for $k \in \mn$ and $q \in [1,\infty]$. Moreover, let us agree that in the following  $\Omega$ will denote some non-empty \textit{open} subset of $\mr^\nu$.
\begin{thm}\label{unique}
Let $W \in L^p_{loc}(\mr^\nu)$ where $p=\nu/2$ if $\nu\geq 3$, $p>1$ if $\nu=2$ and $p=1$ if $\nu=1$. Moreover, let $u \in H^{2,q}_{loc}(\mr^\nu) \cap L^2_{loc}(\mr^\nu)$ where $q=(2\nu)/(\nu+2)$ if $\nu\geq 2$ and $q=1$ if $\nu=1$, and assume that $u$ is a distributional solution of
\begin{equation}
  \label{eq:dis}
 (-\Delta + W) u = 0
\end{equation}
which is zero a.e. on $\Omega$. Then $u$ is zero a.e. on $\mr^\nu$.
\end{thm}
\begin{rem}
 Clearly, in \cite{jerison} this theorem is formulated for $\nu \geq 2$ only. We have included the (obvious) case $\nu=1$ for completeness.
\end{rem}
\begin{rem}
If $W \in L^p_{loc}$ then the same is true of $W-E$ for every $E \in \mr$. 
\end{rem} 
In the remainder of this section we need the following additional assumption on the real part of the potential (if $\nu \geq 2$):
\begin{itemize}
\item[(A3)] $\Re(V) \in L^p_{loc}(\mr^\nu)$ where $p=\nu/2$ if $\nu \geq 3$ and $p>1$ if $\nu=2$. 
\end{itemize}
The next lemma is borrowed from \cite{amrein} (see the final remark in that paper). We include a sketch of proof for completeness.
\begin{lemma}\label{lem:unique}
Assume (A1) and (A3). For $f \in \dom(-\Delta \dotplus \Re(V))$ and $E \in \mr$ let
$$ (-\Delta \dotplus \Re(V))f = Ef.$$
If $f=0$ a.e. on $\Omega$, then $f=0$  a.e. on $\mr^\nu$.
\end{lemma}
\begin{proof}
We only consider the case $\nu \geq 3$. In view of Theorem \ref{unique} it is sufficient to show that $f \in H_{loc}^{2,(2\nu)/(\nu+2)}(\mr^\nu)$. But $f \in H^{1,2}(\mr^\nu)$ by assumption  and $H^{1,2}(\mr^\nu) \subset L^{2\nu/(\nu-2)}(\mr^\nu)$ by Sobolev embedding. So (A3) and H\"older's inequality imply that $(\Re(V)-E)f \in L_{loc}^{(2\nu)/(\nu+2)}(\mr^\nu)$ and then the same must be true of $-\Delta f$. But this shows that $f \in H_{loc}^{2,(2\nu)/(\nu+2)}(\mr^\nu)$, see \cite{MR1996773} Theorems 7.9.7 and 4.5.13.
\end{proof}
Here is our main criterion for the absence of boundary eigenvalues.
\begin{thm}\label{thm:last}
Assume (A1)-(A3). If $H$ has a boundary eigenvalue with imaginary part $b$, then for every non-empty open set $\Omega \subset \mr^\nu$ the set
$$ \{ x \in \Omega : \Im(V(x))=b\}$$
has non-zero Lebesgue measure.
\end{thm} 
\begin{proof}
Let $a+ib$ be a boundary eigenvalue of $H$ with corresponding eigenfunction $f$. Suppose there exists a non-empty open set $\Omega$ such that  $$A :=\{x \in \Omega : \Im(V(x))=b\}$$ has Lebesgue measure zero. Then from Proposition \ref{prop:2} we know that $$B:=\{x : \Im(V(x))f(x) \neq bf(x)\}$$ has Lebesgue measure zero as well. Since $\{x \in \Omega: f(x)\neq 0\}$ is a subset of $A \cup B$, this shows that $f=0$ a.e. on $\Omega$.  But we also have $(-\Delta \dotplus \Re(V))f=af$, so Lemma \ref{lem:unique} implies that $f=0$ a.e. on $\mr^\nu$, a contradiction.
\end{proof}
\begin{rem}
(i) The condition that $\{ x \in \Omega: \Im(V(x))=b\}$ has non-zero Lebesgue measure for every non-empty open set $\Omega \subset \mr^\nu$ means that $\{x : \Im(V(x))=b\}$ is \textit{metrically dense} in $\mr^\nu$ (with respect to Lebesgue measure), see \cite{MR0072928}. This is certainly a very restrictive condition (for instance, it requires that for every Lebesgue null set $N \subset \mr^\nu$ the set $ \{ x \in \mr^\nu \setminus N : \Im(V(x))=b\}$ is dense in $\mr^\nu$). However, we note that this condition can be satisfied simultaneously for two different $b$'s and so (in isolation) does neither imply that all boundary eigenvalues must have the same imaginary part nor that $\Im(V)$ is constant a.e. on $\mr^\nu$.  For instance, this follows from the fact that $\mr^\nu$ can be partitioned into two disjoint metrically dense sets $A_1,A_2$ (see \cite{MR0072928} for a much more general result), by choosing $\Im(V) = b_1\chi_{A_1}+b_2\chi_{A_2}$, where $b_1 \neq b_2$ and $\chi_{A_i}$ denotes the characteristic function of $A_i$.\\
(ii) On the other hand, we are currently not aware of an example of a Schr\"odinger operator with boundary eigenvalues when the imaginary part of the potential is not constant a.e. and the results presented below seem to suggest that such an example (if it exists) might be quite difficult to obtain.\\
(iii) $H=-\Delta \dotplus V$ can of course have boundary eigenvalues if $\Im(V)$ is constant a.e., since in this case it is just a selfadjoint Schr\"odinger operator shifted by a complex constant.
\end{rem}
 Let us further indicate the restrictiveness of  Theorem \ref{thm:last} by considering some corollaries (we always assume (A1)-(A3)). 
 \begin{rem}
  Recall that a measurable function $\tilde{V}: \mr^\nu \to \mc$ is a \textit{representative} of $V$ if these two functions coincide almost everywhere on $\mr^\nu$.
 \end{rem}
\begin{cor}\label{cor:final}  
Let $\tilde{V}$ be a representative of $V$. 
\begin{enumerate}
\item[(i)] If  $b:=\lim_{x \to x_0} \Im(\tilde{V}(x)) \in \mr$ exists for some $x_0 \in \mr^\nu$, then any boundary eigenvalue of $H$ must have imaginary part $b$.
\item[(ii)] If for some $x_0,x_1 \in \mr^\nu$ the limits $\lim_{ x \to x_0} \Im(\tilde{V}(x))$ and $\lim_{x \to x_1} \Im(\tilde{V}(x))$ exist and don't coincide, 
then $H$ has no boundary eigenvalues. 
\end{enumerate}
Furthermore, (i) and (ii) remain valid if $\lim\limits_{x \to x_0} \Im(\tilde{V}(x))$ is replaced with $\lim\limits_{\|x\| \to \infty} \Im(\tilde{V}(x))$.
\end{cor}
\begin{proof}
Let $\tilde{V} : \mr^\nu \to \mc$ with $\tilde{V}(x)=V(x)$ for all $x \in \mr^\nu \setminus N$, where $N$ is a null set. To prove (i) note that if $\lim_{x \to x_0} \Im(\tilde{V}(x)) = b$, then for every $d \neq b$ we can find some $\delta > 0$ such that $\Im(\tilde{V}(x)) \neq d$ for $\|x-x_0\|<\delta$. But then
$$ \{ x : \|x-x_0\| < \delta \text{ and } \Im(V(x))=d \}$$ must be a subset of $N$ and so has Lebesgue measure zero. Now Theorem \ref{thm:last} (applied to the open set $\{x : \|x-x_0\| < \delta\}$) shows that $H$ will have no boundary eigenvalues with imaginary part $d$. A similar argument can be used when $\lim_{\|x\| \to \infty} \Im(\tilde{V}(x)) =b$. Finally, Part (ii) is an immediate consequence of Part (i).
\end{proof}
\begin{rem}
In many applications one is interested in the case where $V(x)$ tends to $0$ for $\|x\| \to \infty$. Here the spectrum of $H$ consists of $[0,\infty)$ and a possible discrete set of eigenvalues which can accumulate at $[0,\infty)$ only. The previous corollary shows that none of the non-real eigenvalues of $H$ will be a boundary eigenvalue.
\end{rem}
From Corollary \ref{cor:final} we immediately obtain
\begin{cor}\label{cor:nexttolast}
Suppose that $V$ has a representative $\tilde{V}$ whose imaginary part is continuous at $x_0$ and $x_1$ with $\Im(\tilde{V}(x_0)) \neq \Im(\tilde{V}(x_1))$. Then $H$ has no boundary eigenvalues.
\end{cor}
% \begin{proof}
% Let $g : \Omega \to \mr$ be continuous and non-constant with $\Im(V(x))=g(x)$ for $x \in \Omega \setminus N$, where $N$ is a null set. Then for every $b \in \mr$ the sets  $\Omega_b:=\{x \in \Omega : g(x) \neq b \}$ are open and non-empty and we have
% $ \{ x \in \Omega_b : \Im(V(x))=b \} \subset N$, showing that $\{ x \in \Omega_b : \Im(V(x))=b \}$ has Lebesgue measure zero. Now use Theorem \ref{thm:last}.
% \end{proof} 
% In the following we use the expression
% "$\Im(V(x)) \to 0$ for $|x| \to \infty$"
% to mean that for every $\eps>0$ there exists $n \in \mn$ such that $$ \{ x : |x| > n \text{ and } |\Im(V(x))| \geq \eps\}$$ has Lebesgue measure zero.
In the following we would like to mention one further condition which implies that all boundary eigenvalues must have the same imaginary part.
\begin{rem}
  We say that an eigenfunction $u$ of $H$ is continuous if it has a continuous representative.
\end{rem}
\begin{thm}
  Assume (A1)-(A3). Suppose that $H$ has a boundary eigenvalue with a continuous eigenfunction. Then all boundary eigenvalues of $H$ must have the same imaginary part.
\end{thm}
\begin{proof}
  Let $a+ib$ be a boundary eigenvalue of $H$ and let $f$ denote a continuous representative of a corresponding eigenfunction. Assume there exists another boundary eigenvalue $c+id$, with $d\neq b$. By Proposition \ref{prop:2} we have $M_{\Im(V)}f=bf$, so 
$$ A:=\{ x : \Im(V(x))f(x) \neq b f(x)\}$$
has Lebesgue measure zero. Moreover, the continuity of $f$ implies that $\{ x: f(x) \neq 0\}$ is open (and non-empty), so we obtain from Theorem \ref{thm:last} (applied to $c+id$) that
$$ B:= \{ x  : f(x) \neq 0 \text{ and } \Im(V(x))= d \}$$
has non-zero Lebesgue measure. On the other hand, since $b \neq d$ we have $B \subset A$. This leads to a contradiction.
\end{proof}
\begin{cor}
Let $\nu=1$ and assume (A1)-(A2). Then all boundary eigenvalues of $H$ (if any) will have the same imaginary part.
\end{cor}
\begin{proof}
  This follows from the fact that $\dom(H) \subset H^{1,2}(\mr) \subset C_0(\mr)$.
\end{proof}
\begin{cor} 
  Assume (A1)-(A3). If all eigenfunctions of $-\Delta \dotplus \Re(V)$ are continuous, then all boundary eigenvalues of $H$ (if any) will have the same imaginary part.
\end{cor}
\begin{proof}
  This is a consequence of Proposition \ref{prop:2}.
\end{proof}
\begin{rem}
  We do not know whether (A1) and (A3) are sufficient for the continuity of the eigenfunctions of $-\Delta \dotplus \Re(V)$ (if $\nu \geq 2$).  However, we note that some of the sufficient conditions for (A1) (like $|V|$ being in the Kato-class, see \cite{MR670130}) do imply this fact, so for these potentials all boundary eigenvalues of $H$ must have the same imaginary part.
\end{rem}
Let us conclude this section by (re-)stating the two main open problems which have arisen in the above considerations:
\begin{enumerate}
\item[(i)] Given (A1)-(A3), can $H=-\Delta \dotplus V$ have boundary eigenvalues if $\Im(V)$ is not constant almost everywhere?
\item[(ii)] If the answer to the first problem is yes, is it possible for $H$ to have boundary eigenvalues with different imaginary parts?
\end{enumerate}

\section{Acknowledgments}
 
I would like to thank  M. Demuth and, especially, G. Katriel for a variety of helpful remarks. Moreover, my thanks go to an anonymous referee whose comments and suggestions 
helped to improve this article.

\bibliography{bibliography}

\def\cprime{$'$}
\begin{thebibliography}{10}

\bibitem{MR1819914}
A.~A. Abramov, A.~Aslanyan, and E.~B. Davies.
\newblock Bounds on complex eigenvalues and resonances.
\newblock {\em J. Phys. A}, 34(1):57--72, 2001.

\bibitem{amrein}
W.~O. Amrein, A.-M. Berthier, and V.~Georgescu.
\newblock {$L^{p}$}-inequalities for the {L}aplacian and unique continuation.
\newblock {\em Ann. Inst. Fourier (Grenoble)}, 31(3):vii, 153--168, 1981.

\bibitem{MR2481997}
A.~Borichev, L.~Golinskii, and S.~Kupin.
\newblock A {B}laschke-type condition and its application to complex {J}acobi
  matrices.
\newblock {\em Bull. Lond. Math. Soc.}, 41(1):117--123, 2009.

\bibitem{Damanik05}
D.~Damanik, R.~Killip, and B.~Simon.
\newblock Schr\"odinger operators with few bound states.
\newblock {\em Comm. Math. Phys.}, 258(3):741--750, 2005.

\bibitem{b_Davies}
E.~B. Davies.
\newblock {\em Linear operators and their spectra}, volume 106 of {\em
  Cambridge Studies in Advanced Mathematics}.
\newblock Cambridge University Press, Cambridge, 2007.

\bibitem{MR2559715}
M.~Demuth, M.~Hansmann, and G.~Katriel.
\newblock On the discrete spectrum of non-selfadjoint operators.
\newblock {\em J. Funct. Anal.}, 257(9):2742--2759, 2009.

\bibitem{MR2163601}
I.~Egorova and L.~Golinskii.
\newblock On the location of the discrete spectrum for complex {J}acobi
  matrices.
\newblock {\em Proc. Amer. Math. Soc.}, 133(12):3635--3641 (electronic), 2005.

\bibitem{MR0072928}
P.~Erd{\H{o}}s and J.~C. Oxtoby.
\newblock Partitions of the plane into sets having positive measure in every
  non-null measurable product set.
\newblock {\em Trans. Amer. Math. Soc.}, 79:91--102, 1955.

\bibitem{Frank11}
R.~L. Frank.
\newblock Eigenvalue bounds for {S}chr\"odinger operators with complex
  potentials.
\newblock {\em Bull. Lond. Math. Soc.}, 43(4):745--750, 2011.

\bibitem{MR2260376}
R.~L. Frank, A.~Laptev, E.~H. Lieb, and R.~Seiringer.
\newblock Lieb-{T}hirring inequalities for {S}chr\"odinger operators with
  complex-valued potentials.
\newblock {\em Lett. Math. Phys.}, 77(3):309--316, 2006.

\bibitem{FLS11}
R.~L. Frank, A.~Laptev, and R.~Seiringer.
\newblock A sharp bound on eigenvalues of {S}chr\"odinger operators on the
  halfline with complex-valued potentials.
\newblock In {\em Spectral Theory and Analysis}, volume 214 of {\em Oper.
  Theory Adv. Appl.}, pages 39--44. Birkh\"auser Verlag, Basel, 2011.

\bibitem{MR2367876}
L.~Golinskii and S.~Kupin.
\newblock Lieb-{T}hirring bounds for complex {J}acobi matrices.
\newblock {\em Lett. Math. Phys.}, 82(1):79--90, 2007.

\bibitem{Hansmann11}
M.~Hansmann.
\newblock An eigenvalue estimate and its application to non-selfadjoint
  {J}acobi and {S}chr\"odinger operators.
\newblock {\em Lett. Math. Phys.}, 98(1):79--95, 2011.

\bibitem{HK10}
M.~Hansmann and G.~Katriel.
\newblock {Inequalities for the eigenvalues of non-selfadjoint {J}acobi
  operators.}
\newblock {\em Complex Anal. Oper. Theory}, 5(1):197--218, 2011.

\bibitem{MR0200725}
S.~Hildebrandt.
\newblock \"{U}ber den numerischen {W}ertebereich eines {O}perators.
\newblock {\em Math. Ann.}, 163:230--247, 1966.

\bibitem{MR1996773}
L.~H{\"o}rmander.
\newblock {\em The analysis of linear partial differential operators. {I}}.
\newblock Classics in Mathematics. Springer-Verlag, Berlin, 2003.
\newblock Distribution theory and Fourier analysis, Reprint of the second
  (1990) edition [Springer, Berlin; MR1065993 (91m:35001a)].

\bibitem{jerison}
D.~Jerison and C.~E. Kenig.
\newblock Unique continuation and absence of positive eigenvalues for
  {S}chr\"odinger operators.
\newblock {\em Ann. of Math. (2)}, 121(3):463--494, 1985.
\newblock With an appendix by E. M. Stein.

\bibitem{kato}
T.~Kato.
\newblock {\em Perturbation theory for linear operators}.
\newblock Classics in Mathematics. Springer-Verlag, Berlin, 1995.
\newblock Reprint of the 1980 edition.

\bibitem{MR2540070}
A.~Laptev and O.~Safronov.
\newblock Eigenvalue estimates for {S}chr\"odinger operators with complex
  potentials.
\newblock {\em Comm. Math. Phys.}, 292(1):29--54, 2009.

\bibitem{MR2198326}
V.~G. Maz'ya and I.~E. Verbitsky.
\newblock Infinitesimal form boundedness and {T}rudinger's subordination for
  the {S}chr\"odinger operator.
\newblock {\em Invent. Math.}, 162(1):81--136, 2005.

\bibitem{MR2651940}
O.~Safronov.
\newblock Estimates for eigenvalues of the {S}chr\"odinger operator with a
  complex potential.
\newblock {\em Bull. Lond. Math. Soc.}, 42(3):452--456, 2010.

\bibitem{MR2596049}
O.~Safronov.
\newblock On a sum rule for {S}chr\"odinger operators with complex potentials.
\newblock {\em Proc. Amer. Math. Soc.}, 138(6):2107--2112, 2010.

\bibitem{MR670130}
B.~Simon.
\newblock Schr\"odinger semigroups.
\newblock {\em Bull. Amer. Math. Soc. (N.S.)}, 7(3):447--526, 1982.

\bibitem{MR1711536}
G.~Teschl.
\newblock {\em Jacobi operators and completely integrable nonlinear lattices},
  volume~72 of {\em Mathematical Surveys and Monographs}.
\newblock American Mathematical Society, Providence, RI, 2000.

\bibitem{MR1443208}
R.~Webster.
\newblock {\em Convexity}.
\newblock Oxford Science Publications. The Clarendon Press Oxford University
  Press, New York, 1994.

\end{thebibliography}
\bibliographystyle{plain}

\end{document}